\documentclass[11pt]{amsart}

\usepackage{palatino}
\usepackage[all]{xy}
\usepackage{amssymb}
\usepackage{latexsym}
\usepackage{amscd}
\usepackage{color}

\makeatletter \@addtoreset{equation}{section}

\makeatother \makeatletter

\makeatletter \@addtoreset{equation}{section}

\makeatother \makeatletter

\def \<{\langle}
\def \>{\rangle}

\def \a{\alpha }

\newtheorem{theorem}{Theorem}[section]
\newtheorem{corollary}{Corollary}[section]
\newtheorem{lemma}{Lemma}[section]
\newtheorem{conjecture}{Conjecture}[section]

\newtheorem{remark}{Remark}[section]

\newtheorem{proposition}{Proposition}[section]

\newcommand{\cV}{\mathcal V}
\newcommand{\scV}{\mathcal {SV}}

\newcommand{\bea}{\begin{eqnarray}}
\newcommand{\eea}{\end{eqnarray}}
\newcommand{\be}{\begin {equation}}
\newcommand{\ee}{\end{equation}}

\newcommand{\nn}{\nonumber \\}
\newcommand{\ver}{L(-\tfrac{4}{3}\Lambda_0)}

\newcommand{\logmin}{\mathcal{W}(p,p')}
\newcommand{\slogmin}{\mathcal{SW}(p,p')}

\newcommand{\Z}{\Bbb Z}

\newcommand{\Zp}{{\Bbb Z}_{>0} }

\newcommand{\N}{{\Bbb Z}_{\ge 0} }
\newcommand{\C}{\Bbb C}

\newcommand{\la}{\langle}
\newcommand{\ra}{\rangle}

\newcommand{\triplet}{\mathcal{W}(p)}
\newcommand{\striplet}{\mathcal{SW}(m)}
\newcommand{\hf}{\mbox{$\frac{1}{2}$}}
\newcommand{\thf}{\mbox{$\frac{3}{2}$}}

\begin{document}

\title[Lattice construction of logarithmic modules]{Lattice construction of logarithmic modules for certain vertex algebras}

\date{}

\author{Dra\v zen Adamovi\' c}

\address{Department of Mathematics, University of Zagreb, Croatia}
\email{adamovic@math.hr}

\author{Antun Milas}

\address{Department of Mathematics and Statistics,
University at Albany (SUNY),  \\ Albany, NY 12222}
\email{amilas@math.albany.edu}

 \subjclass[2000]{
Primary 17B69, Secondary 17B67, 17B68, 81R10}

\keywords{vertex  algebras, W-algebras, lattice vertex algebras;  logarithmic representations; logarithmic intertwining operators; triplet vertex algebras; quantum groups; affine Lie  algebras}

 \markboth{Dra\v zen Adamovi\' c and Antun Milas} {Vertex
superalgebras}
\bibliographystyle{amsalpha}

\begin{abstract}
A general method for constructing logarithmic modules in vertex
operator algebra theory is presented. By utilizing this approach, we
give explicit vertex operator construction of certain indecomposable
and logarithmic modules for the triplet vertex algebra $\triplet$
and for other subalgebras of lattice vertex algebras and their $N=1$
super extensions.  We analyze in detail indecomposable modules
obtained in this way, giving further evidence for the conjectural equivalence
between the category of $\triplet$-modules and the category of modules for the restricted quantum group
$\overline{\mathcal{U}}_q(sl_2)$, $q=e^{\pi i /p}$. We also construct
logarithmic representations for a certain affine vertex operator
algebra at admissible level realized in \cite{A-2005}. In this way
we prove the existence of the logarithmic representations predicted
in \cite{G}. Our approach enlightens related logarithmic
intertwining operators among indecomposable modules, which we also
construct in the paper.

\end{abstract}

 \maketitle
\tableofcontents

\section{ Introduction}

An important problem in vertex algebra theory is to describe
representations of irrational $C_2$-cofinite vertex superalgebras,
in particular the triplet vertex algebras $\triplet$, $p \geq 2$.
The triplet vertex algebra $\triplet$ is interesting for several
reasons: for instance, its representation theory is expected to be
identical to the one of $\overline{\mathcal{U}}_q(sl_2)$, where $q$
is root of unity (see \cite{FGST-triplet}, \cite{FGST-triplet2}). In
fact, it was the triplet algebra (and especially early contributions
\cite{Flohr-mod} and \cite{GaKa}) that gave motivation for studying
general $C_2$-cofinite vertex algebras (see \cite{Hu}, \cite{Miy},
etc.).

The present work is a continuation of the series of papers
\cite{AdM-triplet}-\cite{AdM-ptraces} in which we studied the triplet
and super-triplet vertex algebras $\triplet$ and $\striplet$  and
their representation theory. By using the theory of Zhu's algebra
and modular invariance we proved that these vertex algebras admit
logarithmic representations, i.e., the representation on which
the Virasoro  generator $L(0)$ does not act semisimply.  On the other hand, irreducible and
ordinary modules for $\triplet$ and $\striplet$ can be constructed
explicitly by using lattice vertex algebras and their modules. This
raises a question whether logarithmic modules can be also
constructed via lattice vertex algebras, even though lattice vertex
algebras do not admit  logarithmic modules. In contrast, for the
Heisenberg and singlet vertex algebras (see \cite{AdM-2007} and
\cite{M2}) logarithmic modules are easily constructed by deforming
the action of the zero modes of the Heisenberg algebra. Thus, we
believe it is important to present explicit construction of
logarithmic representations for the triplet.

In this paper we present a general construction of logarithmic modules for vertex (super)algebras obtained as
(intersections of) kernels of "screening" operators (subject to some
additional properties). Vertex algebras obtained in this fashion are
also known as $\mathcal{W}$-algebras, but we stress that
$\mathcal{W}$-algebras considered in this paper do not arise from
the usual quantum reduction as in \cite{Ar}, \cite{DK}, \cite{KW2},
etc.

Although the idea of deforming CFTs with the use of screening
operators has been considered in the literature on LCFT (especially
\cite{FFHST}), our construction is new and covers models not
previously considered. For instance, we are able to explicitly
construct  certain logarithmic representations for $\triplet$ for every $p$,
while in the literature the only explicit construction we are aware
of is for $p=2$
 (cf. \cite{Abe},\cite{GaKa}).
Moreover, methods used to construct  logarithmic modules are very useful to
analyze general indecomposable $\triplet$-modules (see Section 5 and 6).

We also apply our method to other vertex superalgebras of interest,
such as the supertriplet $\striplet$ introduced in
\cite{AdM-striplet}. Moreover, we construct certain logarithmic
representations (cf. Theorem \ref{log-modules-triplet} ) of the
vertex algebras $\mathcal{W}(p,p')$, which are extensions of the
$(p,p')$-minimal models for the Virasoro algebra introduced in
\cite{FGST-log}. We also present a construction of a new family of
$N=1$ vertex operator superalgebras, called $\mathcal{SW}(p,p')$,
which are extensions of the $N=1$ super minimal $(p,p')$-models,
with $(p,\frac{p-p'}{2})=1$. This extension combines into the
following diagram
$$V_{1/2}(c_{p,p'},0) \subset \mathcal{SW}(p,p') \subset V_L \otimes F,$$
where $V_{1/2}(c_{p,p'},0) $ is the (universal) Neveu-Schwarz vertex
superalgebra of central charge $c_{p,p'}$, $F$ is the free fermion
vertex operator and $V_L$ is the lattice vertex superalgebra based
on the integral lattice $L=\mathbb{Z} \alpha$, $ \langle
\alpha,\alpha \rangle=pp'$.

The present work also contributes to better understanding of
logarithmic representations of a certain affine vertex algebra on an
admissible level. Previously, the first author \cite{A-2005} had
obtain an explicit realization of the affine vertex operator algebra
$\ver$ (admissible module for $A_1^{(1)}$ \cite{KW}, \cite{KW2}) and
related modules. In Section 8  we shall construct certain
logarithmic $\ver$--representations. These logarithmic
representations are important for description of extensions of
certain weak modules and analysis of fusion product of admissible
representations. In particular, we present the free field
realization of the logarithmic module ${\mathcal R}_{-1/3}$ which
appeared in \cite{G} and \cite{Ga}. It is also interesting that
logarithmic representations for affine vertex operator algebras
which we have constructed are not $\N$--graded and therefore they
cannot be detected by using theory of Zhu's algebras.

It is well known now that logarithmic modules in vertex algebra
theory are closely related to logarithmic intertwining operators
  \cite{Hu}, \cite{HLZ}, \cite{M1}. As in the ordinary case, logarithmic  intertwining
operators are hard to construct explicitly. In the last part of the
paper we show that already the module map for logarithmic modules
considered earlier give rise to logarithmic intertwining operators.
This is obtained by combining the theory of simple currents
\cite{Li-sc} with our method (cf. Theorems \ref{log-int-1} and
\ref{log-int-2}).

\section{Construction of logarithmic modules}

\label{konstrukcija}

In this part we assume some familiarity with vertex (super)algebras
as in \cite{Kac} and \cite{LL}.

Assume that $(V,Y,{\bf 1}, \omega)$ is a vertex operator
superalgebra, with the parity decomposition $V=V_0 \oplus V_1$,
together with the vertex operator map $Y( \cdot, x)$,
$Y(u,x)=\sum_{n \in \mathbb{Z}} u_n x^{-n-1}$, such that
$$Y(\omega,x)=\sum_{m \in \mathbb{Z}} L(m)x^{-m-2},$$
where the operators $L(m), m \in \mathbb{Z}$ close the Virasoro algebra.

Let $v \in V$
be an even vector such that:
\bea
&& [v_n, v_m] = 0 \quad \forall n,m \in {\Z}, \label{rel-c-1} \\
&& L(n) v = \delta_{n,0} v \quad \forall \ n \in {\N},
\label{rel-c-2}
\eea
so that $v$ is of conformal weight one.

Define
\be \label{delta-main}
\Delta(v,x) = x^{v_0} \exp \left( \sum_{n=1} ^{\infty}
\frac{v_n}{-n}(-x)^{-n} \right).
\ee

If $v_0$ acts semisimply on $V$ then the expression $x^{v_0} w$, where $w$ is an eigenvector
for $v_0$ corresponding to the eigenvalue $\lambda$, is defined as $x^{\lambda} w$.
But (\ref{delta-main}) is ambiguous if $v_0$ does not act
semisimply on $V$ (however, see also Chapter 9).

The following easy but fundamental observation will be of great use in the paper.


\begin{theorem}
\label{gen-const-log} Assume that $V$ is a vertex operator
superalgebra and that $v \in V$ is an even vector which  satisfies
conditions (\ref{rel-c-1}) and (\ref{rel-c-2}). Let $\overline{V}$
be the vertex subalgebra of $V$ such that $\overline{V} \subseteq
\mbox{Ker}_V v_0$.

Assume that $(M,Y_M)$ is
a weak $V$--module (resp. $\sigma$--twisted weak $V$--module).
Define the pair $(\widetilde{M}, \widetilde{Y}_{\widetilde{M} })$
such that
$$\widetilde{M} = M \quad \mbox{as a vector space}, $$
$$ \widetilde{Y}_{\widetilde{M} } (a, x) = Y_{M} (
\Delta(v,x) a, x) \quad \mbox{for} \ a \in \overline{V}. $$
 Then  $(\widetilde{M},
\widetilde{Y}_{\widetilde{M} } )$ is a weak $
\overline{V}$--module (resp. $\sigma$--twisted weak
$\overline{V}$--module).
\end{theorem}
\begin{proof} The proof is essentially in \cite{Li-sc}  and
\cite{DLM-1996}. There the same statement was proven in the case
when $v_0$ is semisimple and without any restriction.  But the same
conclusion obviously holds for $\overline{V}$, because $v_0$ acts
semisimply on $\overline{V}$.
\end{proof}

\begin{remark} In the   case when $v_n$, $n \in {\Z}$, satisfy
the commutation relations for the Heisenberg algebra and $v_0$ acts
semisimply on ${V}$ with integral weights, $\Delta(v,x)$ can be used
for construction of simple current modules (for details see
\cite{Li-sc}, \cite{DLM-1996}).
\end{remark}

Among all (weak) $V$-modules, logarithmic modules are of great interest in this work.
The next result gives a general method for their construction.

\begin{theorem} \label{main}
Assume that $(M,Y_M)$ is a $V$--module such that $L(0)$ acts
semisimply on $M$. Then $(\widetilde{M},
\widetilde{Y}_{\widetilde{M} })$ is logarithmic
$\overline{V}$--module if and only if $v_0$ does not act
semisimply on $M$.
\end{theorem}
\begin{proof} Let
$$ \widetilde{L(x)}= \sum_{n \in {\Z} } \widetilde{L(n)} x^{-n-2}
= \widetilde{Y}_{\widetilde{M} } (\omega,x). $$
Then
$$ \widetilde{L(x)} = Y_M ( \omega + x^{-1} v, x) = L(x) + x^{-1}
Y(v,x), $$
which implies $ \widetilde{L(n)} = L(n) + v_n$, and in particular
$$ \widetilde{L(0)} = L(0) + v_0. $$
Since $L(0)$ and $v_0$ both preserve graded components and the operator $L(0)$ acts as a scalar on each graded
component, it is clear that $\widetilde{L(0)}$ acts nonsemisimply if and only if $v_0$ acts nonsemisimply.
The proof follows. \end{proof}

\begin{remark}
We should mention that in \cite{FFHST}, the authors presented a
related construction of logarithmic modules by using logarithmic
deformation and the extension of the chiral algebra. We think that
it will be interesting to relate these two constructions of
logarithmic representations; we also refer the reader to Section 9,
which is in the spirit of \cite{FFHST}.
\end{remark}

\section{The vertex algebras ${\mathcal W}(p,p')$ and $\triplet$}

Define the lattice  $$L= {\Z} \a, \quad \la \a , \a \ra = 2 p p',
$$
where $p, p' \in {\Zp}$, $p \ne p'$, $p \ge 2$. Let $V_L$ be the associated
(rank one) lattice vertex algebra with the vertex operator map $Y$ (cf. \cite{LL}).

The vertex algebra $V_L$ is a
 subalgebra of the generalized vertex algebra $V_{
 \widetilde{L}}$, where $$\widetilde{L} = {\Z} \frac{\a}{2p p'}$$
 is the dual lattice.
 Let us denote the (generalized) vertex operator map in $V_{
 \widetilde{L}}$ again by  $Y$.

 Define
 $$ \omega = \frac{1}{4 p p'} \a(-1) ^2  + \frac{p-p'}{2 p p'
 }\a(-2).$$
 Then $\omega$ is a conformal vector with central charge
 $c_{p,p'}=1-\frac{6(p-p')^2}{pp'}$. The associated screening operators are
 $$ Q = e^{{\a} /{p'}}_0 = \mbox{Res}_x Y( e^{{\a}/{p'}}, x),
 \quad \widetilde{Q} = e^{-{\a} /{p}}_0 = \mbox{Res}_x Y( e^{- {\a}/{p}}, x).
 $$
It is known (cf. also \cite{A-2003}) that we have  $$[Q,
\widetilde{Q}] = 0.$$

The screening operators $Q$ and $\widetilde{Q}$ enable us to define
certain vertex subalgebras of $V_L$. Define
$$ \overline{V_L} = \mbox{Ker}_{V_L} \widetilde{Q}, \quad \overline{\overline{V_L}} = \mbox{Ker}_{V_L}
Q, \quad \mathcal{W}(p,p') = \overline{V_L} \cap
\overline{\overline{V_L}}. $$
Then $\overline{V_L}$, $\overline{\overline{V_L}}$ and
$\mathcal{W}(p,p')$ are vertex operator algebras with central
charge $c_{p,p'}$. The vertex algebra $\mathcal{W}(p,p')$ was introduced by Feigin et al. in \cite{FGST-log}.

It can be shown (by using Felder's complex or rationality of
$L(c_{p,p'},0)$) that
$$V(c_{p,p'},0) \subset \logmin,$$
where $V(c_{p,p'},0)$ is the universal Virasoro vertex algebra of
central charge $c_{p,p'}$ (generalized Verma module) and
$L(c_{p,p'},0)$ is its simple quotient.

It is not clear to us how to find a strongly generated set for
$\mathcal{W}(p,p')$ in general, but for $p'=2$ we recently made some progress \cite{AdM-wp2}. Then,  motivated mainly by \cite{FGST-log} and \cite{AdM-wp2}, we
do expect the following conjecture to be true (cf. also \cite{AdM-wp2}).

\begin{conjecture} The vertex algebra $\mathcal{W}(p,p')$ is $C_2$-cofinite with $2pp'+\frac{(p-1)(p'-1)}{2}$ inequivalent irreducible modules.
\end{conjecture}

\begin{remark}
In the case $p' =1$, $\overline{V_L}= \triplet$ is the triplet
vertex algebra with central charge $c_{p,1}$. In this case
$$\overline{\overline{V_L}} = \mathcal{W}(p,1) =
L^{Vir}(c_{p,1},0). $$
\end{remark}

\section{Extended vertex algebras $V_L \oplus V_{L-\a / p}$ and $V_L \oplus V_{L+ \a / p'}$}

Observe first that if $V_L $ is of
central charge $c_{p,p'}$, both  $V_{L- \a /p}$ and $V_{L+ \a /
p'}$ are simple  $V_L$--modules with integral weights. Then on $V_L
\oplus V_{L-\a / p}$   (resp.   $V_L \oplus V_{L+ \a / p'}$ ) we
have the structure of a vertex operator algebra (see
\cite{Li-bilinear}).  In this section we shall investigate these
extended vertex operator algebras and its modules. The associated
vertex operators can be reconstructed from the generalized vertex
algebra $(V_{\widetilde{L}}, Y)$.

\vskip 5mm

Define
$$  {\cV}_i =
 V_{ L + \frac{i}{2 p'} \a} \bigoplus V_{ L + \frac{i}{2 p'} \a - \frac{1}{p} \a}, \quad i=0, \dots, 2 p'-1,$$
 and
$$ {\cV}^o_j =
 V_{ L + \frac{j}{2 p} \a} \bigoplus V_{ L + \frac{j}{2 p} \a + \frac{1}{p'} \a}, \quad j=0, \dots, 2 p-1.$$

Then ${\cV}= {\cV}_0$ (resp. ${\cV}^o={\cV}_0^o$) is a vertex
operator algebra and each ${\cV}_i$ (resp. ${\cV}^o_j $) is an
${\cV}$--module (resp. ${\cV}^o$--module). The associated vertex operators $Y_{{\cV}_i}$  are
defined as follows:
$$ Y_{{\cV}_i}  ( u + v, x) (u_i + v_i) = Y (u + v, x) u_i + Y(u,x
) v_i,$$
where $ u \in V_L$, \ $v  \in  V_{L-\frac{1}{p}\a}$, \ $u_i \in V_{
L + \frac{i}{2 p'} \a}$, \ $v_i \in V_{ L + \frac{i}{2 p'} \a -
\frac{1}{p} \a}$. (Note that we require $Y_{{\cV}_i}(v,x)v_i = 0$.)
Similarly we define the vertex operator $Y_{{\cV}^o _j}$.

\begin{remark}
Assume that $V$ is a vertex operator algebra and $M$ is a
$V$--module with integral weights. Then $V \oplus M$ carries the
vertex operator algebra structure (cf. \cite{Li-bilinear}). This
structure in the case $V=V_L$ and $M= V_{L-\a /p}$ coincides with
the vertex operator algebra structure on ${\cV}$ described above.
\end{remark}

We fix $V= {\cV}_0$.
%
%
In this section we shall consider the case $v = e ^{-\a / p}$.
Then
$$[v_n, v_m] = 0 \quad \forall m,n \in {\Z}.$$
But $v_0 = \widetilde{Q}$ is a screening operator and it does not
acts semisimply on ${\cV}_0$. On the other hand $\widetilde{Q}$
acts trivially on $\overline{V_L}$ and we can use $\Delta( v,x)$
to construct weak $\overline{V_L}$--modules. Similarly we can use
$\Delta( e ^{\a /p'},x)$ to construct weak
$\overline{\overline{V_L}}$--modules.

\begin{theorem} \label{log-modules-triplet}

\item[(i)] Assume that $(M, Y_M)$ is a weak ${\cV}$--module. Then
the structure  $$(\widetilde{M}, \widetilde{Y}_{\widetilde{M}
}(\cdot,x) ) := ( M, Y_M( \Delta(e^{-\a /p},x) \cdot,x))$$
%
%
%
  is a weak $ \overline{V_L}$--module. In particular,
$$(\widetilde{{\cV}_i}, \widetilde{Y}_{\widetilde{{\cV}_i} } ) ,
\quad  i=0, \dots, 2 p'-1, $$ are weak $\overline{V_L}$--modules.

\item[(ii)] Assume that $(M, Y_M)$ is a weak ${\cV}^{o}$--module.
Then the structure  $$(\widetilde{M}, \widetilde{Y}_{\widetilde{M}
}(\cdot,x)) := ( M, Y_M( \Delta(e^{\a /p'},x) \cdot,x))$$
%
%
%
  is a weak $ \overline{\overline{V_L}}$--module. In particular,
$$(\widetilde{{\cV}^o_j}, \widetilde{Y}_{\widetilde{{\cV}^o_j} } ) ,
\quad  j=0, \dots, 2 p-1, $$ are weak
$\overline{\overline{V_L}}$--modules.
\end{theorem}

We shall now investigate these $\overline{V_L}$--modules. For
simplicity we set
$\widetilde{Y}=\widetilde{Y}_{\widetilde{{\cV}_i} }$. Let
$$ L(x) = Y(\omega, x) = \sum_{n \in {\Z} } L(n) x^{-n-2}, \quad
\widetilde{L(x)} = \widetilde{Y}(\omega, x) = \sum_{n \in {\Z} }
\widetilde{L(n)} x ^{-n-2}. $$
Then
$$ \widetilde{L(x)} = L(x) +  x^{-1} Y( e^{-\a /p},x), \quad
  $$
$$ \widetilde{L(n)} = L(n) +   e^{-\a /p }_n.$$
In particular,
$$ \widetilde{L(0)} = L(0) +   \widetilde{Q}.$$
Similarly we see that on $\widetilde{{\cV}^o_i}$ we have
$$ \widetilde{L(0)} = L(0) +   Q.$$
Now it is easy to see that modules $\widetilde{{\cV}_i}$ and
$\widetilde{{\cV}^o_i}$ are logarithmic.

\begin{corollary}
\item[(i)] For every $i =0, \dots, 2 p'-1$, $\widetilde{{\cV}_i}$
are logarithmic $\overline{V_L}$--modules.
\item[(ii)] For every $i =0, \dots, 2 p'-1$, $\widetilde{{\cV}_i}$
are logarithmic ${\mathcal W}(p,p')$--modules.
\item[(iii)] For every $i =0, \dots, 2 p-1$,
$\widetilde{{\cV}^o_i}$ are logarithmic ${\mathcal
W}(p,p')$--modules.
\end{corollary}

\begin{remark} Observe that nothing had prevented us to work with more general operators $\Delta( \lambda e^{-\alpha/p},x)$, where
$\lambda \neq 0$. But this modification does not lead to new representations; it merely contributes to the endomorphism
algebra of a (reducible) logarithmic module in question.
\end{remark}

\section{Logarithmic  and indecomposable modules for the triplet vertex algebra
$\triplet$}

\label{triplet}

Now we shall consider the case $p'=1$ and $p \ge 2$. Then
$\overline{V_L}$ is the triplet vertex algebra $\triplet$
investigated in \cite{AdM-triplet}, \cite{FHST}, \cite{Flohr-mod},
\cite{FGST-triplet}, \cite{FGST-triplet2}, etc. For more about the
triplet algebra see \cite{Ka}, \cite{Fu}, \cite{CF}, etc.

We know that $\triplet$ is a vertex subalgebra of $V_L$, strongly generated by
$$ \omega, \quad F= e^{-\a}, \quad H= Q F, \quad E = Q^2 F.$$
Recall also that $\triplet$ has exactly $2p$ irreducible modules:
$$ \Lambda(1), \dots, \Lambda(p), \Pi(1), \dots, \Pi(p).$$

The construction in the previous section provides two logarithmic
$\triplet$--modules: $(\widetilde{\cV}, \widetilde{Y})$ and
$(\widetilde{{\cV}_1},\widetilde{Y} )$.

For  $a \in \triplet$, let
$$\widetilde{Y}(a,x) = \sum_{n \in \Z} \widetilde{a_n} x^{-n-1} .
$$

Define the following operator acting on $\widetilde{{\cV}}$ and
$\widetilde{{\cV}}_1$:
\bea && G = e^{\a}_0 + {\nu}_p e^{\a - \a /p}_{-1}, \qquad \nu_p =
\frac{p}{p-1}. \label{scr-def} \eea

Then the definition of $ \widetilde{Y}(a,x)$ implies that
\bea && \widetilde{a_n} = a_n + \sum_{m =1 } ^{ \infty} \frac{(-1)
^m }{-m} ( e^{-\a /p}_m a)_{n-m} . \label{act-def} \eea

\begin{lemma} \label{G-action} On $\triplet$--modules $\widetilde{{\cV}}$ and
$\widetilde{{\cV}}_1$ we have
\item[(i)]
$ [G, \widetilde{a_n} ] = \widetilde{( Q a)_n},$
for every $n \in {\Z}$ and $a \in \triplet$.

\item[(ii)] $ [ G, \widetilde{L(n)} ]= 0,$
i.e., $G$ is a screening operator which commutes with the action
of the Virasoro algebra on $\widetilde{{\cV}}$ and
$\widetilde{{\cV}}_1$.

\end{lemma}
\begin{proof} It is clear that (i) implies (ii). Let us prove
assertion (i). By using formula (\ref{act-def}) we get:
\bea
[G, \widetilde{a_n} ] &=&   [ Q + {\nu}_p e^{\a - \a /p} _{-1},
\widetilde{a_n}] \nonumber \\
& = & \widetilde{( Q a)_n} + {\nu}_p [ e ^{\a - \a /p}_{-1}, a_n]
+ \sum_{m =1 } ^{ \infty} \frac{(-1) ^m }{-m} ( [Q,  e^{-\a /p}_m]
a)_{n-m} \nonumber \\
& = & \widetilde{( Q a)_n}  +{\nu}_p \sum_{i =0 } ^{ \infty} {-1
\choose i}    (  e^{\a -\a /p}_{i} a)_{n-1-i} +  {\nu}_p \sum_{m
=1 } ^{ \infty} (-1) ^m  (  e^{\a -\a /p}_{m-1} a)_{n-m} \nonumber
\\
& = & \widetilde{( Q a)_n}. \nonumber
 \eea
The proof follows. \end{proof}

\vskip 5mm

For $X \in \{E,F,H\}$, we define
$$ \widetilde{X}(n):= \tilde{X}_{n+2p-2}.$$
We shall consider the following $\triplet$--module:

 $$  \mathcal{P}^+_{p-1}= \triplet . e^{\frac{\a}{2}} \subseteq
\widetilde{{\cV}_1}.$$

\begin{theorem} \label{log-p1}

\item[(i)] The module $\mathcal{P}^+_{p-1}$ is a ${\N}$--graded
logarithmic $\triplet$--module
$$ \mathcal{P}^+_{p-1} = \bigoplus_{n \in \N}  \mathcal{P}^+_{p-1}
(n),$$
with the top component  $\mathcal{P}^+_{p-1} (0)$ of dimension two, spanned by $\mbox{span}_{\C} \{
e^{\frac{\a}{2}}, e^{\frac{\a}{2} -\frac{\a}{p}} \}$. In this basis
$\widetilde{L(0)}$ acts on $\mathcal{P}^+_{p-1} (0)$ as
$$\left[\begin{array}{cc} \frac{2-p}{4} & 1 \\ 0 &   \frac{2-p}{4} \end{array} \right].$$
\item[(ii)] The $\triplet$--module $\mathcal{P}^+_{p-1}$ is
is invariant under the action of $G$, and we have
$$\mathcal{P}^+_{p-1}=\widetilde{\mathcal{V}_1}.$$
\end{theorem}
\begin{proof} By using direct calculation, one can see that
$$ \widetilde{X} (n) v = \widetilde{L(n+1)} v =0,
\quad \mbox{for} \ n \in {\N}
 $$
for $v \in \{ e^{\frac{\a}{2}}, e^{\frac{\a}{2} -\frac{\a}{p}}
\}$.
This implies that $\mathcal{P}^+_{p-1}$ is a ${\N}$--graded module
for $\triplet$.
Now the proof follows from:
\bea
&&\widetilde{L(0)} e^{\frac{\a}{2}} = L(0) e^{\frac{\a}{2}} +
\widetilde{Q}e^{\frac{\a}{2}}=  \frac{2-p}{4} e^{\frac{\a}{2}}
+ e^{\frac{\a}{2} -\frac{\a}{p}}, \nonumber \\
&&\widetilde{L(0)}e^{\frac{\a}{2} -\frac{\a}{p}} = L(0)
e^{\frac{\a}{2} -\frac{\a}{p}} = \frac{2-p}{4} e^{\frac{\a}{2}
-\frac{\a}{p}}.  \nonumber
 \eea
  Let us prove assertion (ii). First we recall the usual (undeformed) action of $\triplet$ on
$V_{L+\alpha/2}$. By using results from \cite{AdM-triplet}, one can
easily see that under this action $V_{L+\alpha/2}$ is a cyclic
module generated by $e^{\alpha/2}$. Moreover   $V_{ L + \a / 2 - \a
/p}$ is a cyclic module generated by $e^{3 /2 \a - \a /p}$. On the
other hand $\mathcal{W}(p) \cdot e^{\alpha/2-\alpha/p}=\Lambda(p-1)
\subset \mathcal{P}^+_{p-1}$. Thus, it is sufficient to show
$e^{3 /2 \a - \a /p} \in \mathcal{P}^+_{p-1}$.
Our Lemma \ref{G-action} directly implies that $\mathcal{P}^+_{p-1}$
is $G$-invariant.

Since $ e ^{-\a /2} \in \mathcal{P}^+_{p-1}$, we conclude that
$$ G^2 e^{-\a /2 } \in \mathcal{P}^+_{p-1}. $$
By direct calculation we have:
$$ G^2 e ^{-\a /2} = 2 \nu_p Q e ^{ \a - \a/p} _{-1} e ^{-\a /2 } = C e
^{ 3 /2 \a - \a /p }$$ for certain non-zero constant $C$. The
constant $C$ can be easily computed from the relation $$Q e ^{ \a -
\a/p} _{-1} e ^{-\a /2 }={\rm Res}_{x_1} {\rm Res}_{x_2} x_2^{-1}
Y(e^{\alpha},x_1)Y(e^{\alpha-\alpha/p},x_2)e^{-\alpha/2}$$ and the
explicit formulas $Y(e^{\alpha},x)$ and $Y(e^{\alpha-\alpha/p},x)$.
The proof follows.
\end{proof}

\vskip 5mm

Define the following submodules of $\mathcal{P}^+_{p-1}$:
\bea
{\mathcal M}_1  & = & {\triplet}. e^{ -\a /2}, \nonumber \\
{\mathcal N}_1 &=& {\triplet}. e^{-\a /2 } + {\triplet}. e ^{3 \a /2
- \a / p} = {\mathcal M}_1 + V_{ L + \a /2 - \a /p}, \nonumber \\
{\mathcal N}_2 &=& {\triplet}. e^{\a /2 - \a /p} \cong
\Lambda(p-1) . \nonumber
 \eea

\begin{lemma} \label{non-split}
We have
\begin{itemize}
\item[(i)]
The $\triplet$--module $\mathcal{M}_1=\mathcal{W}(p) \cdot e^{-\alpha/2}$ defines a non-split
extension
$$0 \rightarrow \Lambda(p-1) \rightarrow \mathcal{M}_1 \rightarrow \Pi(1) \rightarrow 0,$$
of modules $\Pi(1)$ and $\Lambda(p-1)$.
\item[(ii)] The modules $\mathcal{M}_1$ and $V_{L+ \a/2 - \a /p}$ are non-isomorphic.
\end{itemize}

\end{lemma}
\begin{proof}  We clearly have $\mathcal{N}_2 \subset
V_{L+\a/2-\a/p}$.
From Lemma \ref{G-action}, we compute \be \widetilde{E({p-1})}
e^{-\alpha/2} =\widetilde{(Q^2 e^{-\alpha})}_{3p-3}
e^{-\alpha/2}=(ad \ G)^2(e^{-\alpha}_{3p-3}) \cdot e^{-\alpha/2}.
\ee But $$e^{-\alpha}_{3p-3} e^{-\alpha/2}=e^{-\alpha}_{3p-3} G
e^{-\alpha/2}=0,$$ and thus
$$\widetilde{E(p-1)}
e^{-\alpha/2}=\widetilde{F (p-1)}G^2 e^{-\alpha/2}=C
e^{\alpha/2-\alpha/p},$$ for some $C \neq 0$. This implies that
$\mathcal{N}_2 \subset \mathcal{M}_1$.

To prove the assertion it is sufficient to show that:
\bea \label{uvjet2} &&  V_{ L + \a / 2- \a /p} \cap \mathcal{M}_1
= \mathcal{N}_2. \eea

 By using Lemma
\ref{G-action} again we prove:
$$ \widetilde{E(0)} e^{-\a /2} = -2  { 3p -2 \choose 2p-1} Q
e^{-\a /2} + w$$
where $$w = -2 \nu_p { 3p -2 \choose 2p-1} e ^{\a - \a /p}_{-1}
e^{-\a/2} +  \widetilde{F(0)}  G ^2 e ^{-\a/2}. $$
Since $Q w = 0$ we conclude that $w$ is not a subsingular vector
in $V_{ L + \a /2 - \a /p}$ which gives  that $w \in {\mathcal
N}_2$ (here we use the structure of $V_{ L + \a /2 - \a /p}$ as a
module for the Virasoro algebra from \cite{AdM-triplet}).

We also have:

$$\widetilde{H(0)} e^{-\a /2} = -{ 3p -2 \choose 2p-1} e^{-\a/2},
\quad \widetilde{H(0)} Q e^{-\a /2} = { 3p -2 \choose 2p-1} Q
e^{-\a/2} -\frac{1}{2} w. $$

Therefore the top component of the quotient module
$\mathcal{M}_1 / \mathcal{N}_2$ is of dimension two.
Assume that $V_{ L + \a / 2- \a /p} \cap \mathcal{M}_1 \supsetneqq
\mathcal{N}_2 $. Then clearly $$V_{ L + \a / 2- \a /p} \cap
\mathcal{M}_1 = V_{ L + \a / 2- \a /p}  $$ which implies that the
top component of $\mathcal{M}_1 / \mathcal{N}_2$ is
$4$-dimensional. A contradiction.
 Therefore (\ref{uvjet2}) holds. This implies that the quotient
module $\mathcal{M}_1 / \mathcal{N}_2$ is isomorphic to $\Pi(1)$.

For (ii) we only have to show that  $V_{L+\alpha/2-\alpha/p}
\ncong \mathcal{M}_1$. This is done by contradiction. Suppose
there is an isomorphism from $\mathcal{M}_1$ to
$V_{L+\alpha/2-\alpha/p}$. Since $e^{-\alpha/2}$ is a highest
weight vector for the Virasoro algebra this vector has to be
mapped to another highest weight vector in
$V_{L+\alpha/2-\alpha/p}$ of the same conformal weight. Such a
vector is unique up to a scalar and proportional to $e^{3
\alpha/2-\alpha/p}$. But  then $\widetilde{F}_{3p-3} e^{3
\alpha/2-\alpha/p}=e^{\alpha/2-\alpha/p}$, while
$\widetilde{F}_{3p-3}e^{-\alpha/2}=0$, a contradiction.
\end{proof}

The module $\mathcal{M}_1$ in the previous lemma was also predicted in  \cite{FHST}  (see also \cite{FGST-triplet}).

\begin{proposition} \label{filt-3}
The  socle part of the  $\triplet$--module $\mathcal{P}^+_{p-1}$ is
isomorphic to  $ \Lambda(p-1)$. The module $\mathcal{P}^+_{p-1}$ has
semisimple length three  with  the following  semisimple filtration:
\bea \label{filt-1} && \mathcal{P}^+_{p-1} ={\mathcal N}_0 \supset
{\mathcal N}_1 \supset {\mathcal N}_2 \supset {\mathcal N}_3 =0,
\eea
such that
$$ {\mathcal N}_0  / {\mathcal N}_1 \cong \Lambda (p-1), \quad
{\mathcal N}_1 / {\mathcal N}_2 \cong \Pi(1) \oplus \Pi(1), \quad
{\mathcal N}_2  / {\mathcal N}_3 \cong \Lambda (p-1).
$$

\end{proposition}
\begin{proof} Let $W$ be a non-zero submodule of
$\mathcal{P}^+_{p-1}$. By using  action of  $\triplet$ on vectors
from $W$, one sees that
$$ W \cap \mathcal{P}^+_{p-1} (0) \ne \{0\}.$$
Then action of $\widetilde{L(0)}$ on the  top level
$\mathcal{P}^+_{p-1} (0)$ implies that $e^{ \a /2 - \a /p} \in W$.
In this way we have proved that $W$ contains  ${\mathcal N}_2$.
Since $W$ is an arbitrary submodule, we conclude that  the socle
part of $\mathcal{P}^+_{p-1}$ is ${\mathcal N}_2 \cong
\Lambda(p-1)$.

Since  $\mathcal{P}^+_{p-1} / {\mathcal N}_2$ is not semisimple
$\triplet$--module, there are no semisimple filtration of length
two.

In order to prove that (\ref{filt-1}) is a semisimple filtration
of length three, it suffices to show that ${\mathcal N}_1 /
{\mathcal N}_2 $ is semisimple. First we notice that
$V_{ L + \a /2 - \a /p} / \mathcal{N}_2 \cong \Pi(1)$ (cf.
\cite{AdM-triplet}) and $\mathcal{M}_1 / \mathcal{N}_2 \cong
\Pi(1)$ (cf. Lemma \ref{non-split})  which easily implies that
$$(V_{L + \a /2 - \a /p} / \mathcal{N}_2 ) \bigcap \mathcal{M}_1 /
\mathcal{N}_2= \{ 0 \}.$$
In this way we have obtained that
$${\mathcal N}_2 / {\mathcal N}_1 = ( {\mathcal M}_1 + V_{L + \a /2 - \a /p} ) / {\mathcal N}_2 \cong \Pi(1) \oplus \Pi(1). $$

So filtration (\ref{filt-1}) is a semisimple filtration  of
minimal length, and therefore $\mathcal{P}^+_{p-1}$ has semisimple
length three.
\end{proof}

%
%

%
%
%
%
%
%
%
%


In the proof of Proposition  \ref{filt-3} we have shown that every
non-zero submodule of ${\mathcal P}_{p-1} ^+$ intersects top
component non-trivially. Thus, ${\mathcal P}_{p-1} ^+$ is precisely
the $\triplet$--module induced from  the two-dimensional module
${\mathcal P}_{p-1} ^+(0)$ for Zhu's algebra $A(\triplet)$.

\begin{corollary}
Assume that $M$ is a ${\N}$--graded $\triplet$--module such that:
\item[(i)] The top component $M(0)$ is isomorphic to
$\mathcal{P}^+_{p-1}(0)$ as a module for the Zhu's algebra
$A(\triplet)$.

\item[(ii)]  If  $W$ is a non-zero submodule  of $M$, then  $ W
\cap M(0) \ne \{0 \}$.

Then we have $W \cong {\mathcal P}_{p-1} ^+$. In particular, the
module ${\mathcal P}_{p-1} ^{+}$ is self-dual.

\end{corollary}

\begin{proof} The first assertion follows from the theory of Zhu's
algebras and the above discussion. Since  the contragradient
module for ${\mathcal P}_{p-1} ^+$ also has the  top component
isomorphic to ${\mathcal P}_{p-1} ^+(0)$, we easily conclude that
${\mathcal P}_{p-1} ^+$ is a self-dual $\triplet$--module. \end{proof}

\vskip 5mm

 The other logarithmic $\mathcal{W}(p)$-module
$\widetilde{\mathcal{V}}$, which will be denoted by
$\mathcal{P}^-_{p-1}$, has a slightly different structure as
illustrated in our next result (the proof is analogous to the one
in the previous proposition).
\begin{proposition}  We have:
\begin{itemize}
\item[(i)] The module $\mathcal{P}^-_{p-1}$ is of semisimple length three and its socle part is $\Pi(p-1)$.
Moreover, $\mathcal{P}^-_{p-1}$ admits the semisimple filtration
\bea \label{filt-2} && \mathcal{P}^ - _{p-1} ={\mathcal N}^{-} _0
\supset {\mathcal N}^{-} _1 \supset {\mathcal N}^{-} _2 \supset
{\mathcal N}^{-} _3 =0, \eea
such that
$$ {\mathcal N}^{-} _0  / {\mathcal N}^{-} _1 \cong \Pi (p-1), \quad
{\mathcal N}^{-} _1 / {\mathcal N}^{-} _2 \cong \Lambda(1) \oplus
\Lambda(1), \quad {\mathcal N}^{-} _2  / {\mathcal N}^{-} _3 \cong
\Pi (p-1).
$$
\item[(ii)] The first nontrivial $\widetilde{L(0)}$-Jordan block of $\mathcal{P}^-_{p-1}$ opens at the degree one
(i.e., there exists $v \in \mathcal{P}^-_{p-1}(1)$ such that
$\widetilde{L(0)}v=v+w$, where $ w \neq 0$ and
$\widetilde{L(0)}w=w$).
\item[(iii)] If we let  $\widetilde{Y}'( \cdot,
x)=\widetilde{Y}(\Delta(\frac{\alpha(0)}{2},x) \cdot,x)$ , then
$(\widetilde{\mathcal{V}}, \widetilde{Y}'( \cdot,x))$ is
isomorphic to $\widetilde{\mathcal{V}_1}$.
\end{itemize}
\end{proposition}

\begin{remark} \label{conj-log}
Existence of certain logarithmic $\triplet$--modules was proven in
\cite{AdM-triplet} by using modular invariance and the theory of
Zhu's algebras. In Theorem  \ref{log-p1}  we have presented an
explicit realization of a logarithmic module from Proposition 6.6 of
\cite{AdM-triplet}.
\end{remark}

\begin{remark} In \cite{FGST-triplet},\cite{FGST-triplet2} (see also \cite{FHST}) it was conjectured that $\mathcal{W}(p)$ admits $2(p-1)$ indecomposable
(nonisomorphic) logarithmic modules. Although our method gives $2$ logarithmic $\mathcal{W}(p)$-modules for fixed $p$, there is
a strong evidence that remaining logarithmic modules can be also constructed  by using our method. The main difficulty here is to find a correct
replacement for the $Q$ operator in an appropriate extension of $V_L$. This is work in progress.
\end{remark}

\section{Indecomposable $\triplet$-modules and the quantum group \ $\overline{\mathcal{U}}_q(sl_2)$,  \  $q=e^{\pi i / p}$}

In this section discuss conjectural
equivalence between the category of finitely generated
$\triplet$-modules and the category of finite dimensional modules
for the restricted quantum group $\overline{\mathcal{U}}_q(sl_2)$, where $q$ is a primitive $2p$-th root of unity
(cf. \cite{FGST-triplet}, \cite{FGST-triplet2}).

We also discuss some
additional indecomposable $\triplet$-modules.

In \cite{AdM-triplet}, we decomposed the category of ordinary
$\triplet$-modules into $p+1$ blocks.  Two blocks
contain a single irreducible module, and remaining  $p-1$ blocks
contain two irreducible modules each. In fact, similar decomposition
persists at the level of weak $\triplet$-modules, and a proof can be
given along the lines of \cite{AdM-triplet} by analyzing Virasoro
algebra fusion. However, if $p$ is prime, because of Lemma 6.2
\cite{AdM-triplet} and the discussion after the lemma it is clear
that there are no (logarithmic) extensions between different blocks,
so no Virasoro fusion is needed. So let $\mathcal{O}$ be the
category of finitely generated weak $\triplet$-modules, and
$$\mathcal{O}={\bigoplus_{i=1}^{p+1}} \mathcal{O}[h_{i,1}]$$
the corresponding block decomposition.

In the previous section we discussed in more details two logarithmic
modules $\mathcal{P}^+_{p-1} \in {\rm Ob}(\mathcal{O}[h_{p-1,1}])$
and $\mathcal{P}^-_{p-1} \in {\rm Ob} (\mathcal{O}[h_{1,1}])$.
%

%
%
Notice that inside the logarithmic module  $\mathcal{P}^+_{p-1}$ we have
realized  the following non-trivial  extensions of
$\triplet$--modules:
\bea && 0 \rightarrow \Lambda(p-1) \rightarrow \mathcal{M}_1
\rightarrow \Pi(1) \rightarrow 0; \nonumber \\
&& 0 \rightarrow \Lambda(p-1) \rightarrow V_{L + \a /2 - \a /p}
\rightarrow \Pi(1) \rightarrow 0; \nonumber \\
&& 0 \rightarrow \Lambda(p-1) \rightarrow \mathcal{N}_1
\rightarrow \Pi(1) \oplus \Pi(1) \rightarrow 0;\nonumber \\
&& 0 \rightarrow \mathcal{N}_1  \rightarrow \mathcal{P}^+_{p-1}
\rightarrow \Lambda(p-1) \rightarrow 0.\nonumber \eea

 From this
we clearly obtain the following embedding structure for the module
$\mathcal{P}^+_{p-1}$ (a similar diagram can be drawn for
$\mathcal{P}^-_{p-1}$):
 \be \label{diamond}
\xymatrix@C=0.5pc{  &   & \stackrel{\Lambda(p-1)}{\bullet} \ar@/^/[dll] \ar@/_/[drr]   &  &
\\  \stackrel{\Pi(1)}{\bullet} \ar@/^/[rrd] & & &  & \stackrel{\Pi(1)}{\bullet} \ar@/_/[lld]
\\   & &  \stackrel{\Lambda(p-1)}{\bullet} & &
}
\ee
We find the structure of this logarithmic module identical to the
structure of the corresponding projective modules in the case of the
quantum group $\bar{\mathcal{U}}_q(sl_2)$, $q=e^{\pi i/p}$ \cite{FGST-triplet},
\cite{FGST-triplet2}.

From (\ref{diamond}) we also obtain the following indecomposable module (the so-called "wedge"):
\be 
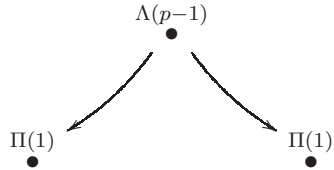
\label{wedge}
\xymatrix@C=0.5pc{  &   & \stackrel{\Lambda(p-1)}{\bullet} \ar@/^/[dll] \ar@/_/[drr]   &  &
\\  \stackrel{\Pi(1)}{\bullet} & & &  & \stackrel{\Pi(1)}{\bullet}
}
\ee
simply by taking the obvious quotient of $\mathcal{P}^+_{p-1}$. But there is also another wedge
of the form:
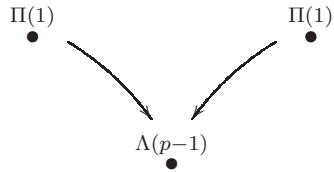
\be \label{wedge-op}
\xymatrix@C=0.5pc{
 \stackrel{\Pi(1)}{\bullet} \ar@/^/[drr] & & &  & \stackrel{\Pi(1)}{\bullet}  \ar@/_/[dll] \\
 &   & \stackrel{\Lambda(p-1)}{\bullet}   &  &
}
\ee

We can now "glue" together several wedges (\ref{wedge}) to obtain:

\be \label{morewedge}
\xymatrix@C=0.5pc{  &   & \stackrel{\Lambda(p-1)}{\bullet} \ar@/^/[dll] \ar@/_/[drr]   &  &   &   & \ \  \cdots  \ \  \ar@/^/[dll] \ar@/_/[drr]   &  &   &   & \stackrel{\Lambda(p-1)}{\bullet} \ar@/^/[dll] \ar@/_/[drr]   &  &
\\  \stackrel{\Pi(1)}{\bullet} & & &  & \stackrel{\Pi(1)}{\bullet}   & & \ \ \cdots \ \ &  & \stackrel{\Pi(1)}{\bullet}  & & &  & \stackrel{\Pi(1)}{\bullet}
}
\ee

and similarly we get

\be \label{loop}
\xymatrix@C=0.5pc{
\stackrel{\Pi(1)}{\bullet}    \ar@/^/[drr]  & & &  & \stackrel{\Pi(1)}{\bullet}  \ar@/_/[dll] \ar@/^/[drr]  & & \ \ \cdots \ \ &  & \stackrel{\Pi(1)}{\bullet} \ar@/_/[dll] \ar@/^/[drr]   & & &  & \stackrel{\Pi(1)}{\bullet}  \ar@/_/[dll]   \\
 &   & \stackrel{\Lambda(p-1)}{\bullet}    &  &   &   & \ \  \cdots  \ \     &  &   &   & \stackrel{\Lambda(p-1)}{\bullet}    &  &
}
\ee


All representations considered here can be constructed explicitly.
For instance (\ref{wedge-op}) is obtained by taking the $L(0)$-semisimple part of $\mathcal{P}_{p-1}^+$ (cf. module $\mathcal{N}_1$ in the previous section).
The representation (\ref{diamond})  is the only logarithmic representation here and conjecturally
this is a projective cover of the irreducible module $\Lambda(p-1)$. The other modules considered here are clearly indecomposable.
We can of course obtain similar modules within the $\mathcal{O}[h_{1,1}]$ block, we only have to switch the roles of $\Lambda$ and $\Pi$.

Another interesting class of representations is obtained as submodules of $\mathcal{P}^+_{p-1}$, where there infinitely many nonisomorphic submodules
parameterized by $[x:y] \in \mathbb{C} \mathbb{P}^1$. These modules can be again constructed explicitly (we omit details here), and are in fact
different Baer sums of extensions obtained from $\mathcal{M}_1$ and $V_{L+\alpha/2-\alpha /p}$. These modules have embedding structure
\be \label{baer}
\xymatrix@C=0.5pc{  &   & \stackrel{\Pi(1)}{\bullet} \ar@/^/[d]^x  \ar@/_/[d]_y   &  &
\\  &&  \stackrel{\Lambda(p-1)}{\bullet} & &
}
\ee

This type of modules are known to exist on the quantum group side.


\begin{remark} The block $\mathcal{O}[h_{1,1}]$ is important for several reasons. Firstly, this is the principal block since it contains the "trivial" module $\mathcal{W}(p)$. But this block is also the "shortest" in the sense that two irreducible modules ($\triplet=\Lambda(1)$ and $\Pi(p-1)$ are at the "conformal distance" one). Thus analysis of
Zhu's algebra $A_1(\triplet)$ ought to be sufficient to classify all indecomposable modules in the block. We plan to study this problem in our future work.
\end{remark}

\section{Logarithmic modules for $\striplet$ and $\slogmin$}

As in the previous section we can now construct logarithmic modules
for the super-triplet vertex algebra $\striplet$ and related extensions of $N=1$ superconformal minimal models $\slogmin$.
To the best of our knowledge, the algebra $\slogmin$ has not been previously studied, so here we give its definition and a few basic  properties omitting proofs that are identical as those in \cite{AdM-striplet}.

Let $p > p'$ be positive integers where $(p,\frac{p-p'}{2})=1$. Consider the rank one lattice $\mathbb{Z} \alpha$, with $\langle \alpha,\alpha \rangle =pp'$, and the vector space$V_L=M(1) \otimes \mathbb{C}[L]$, where $M(1)$ is as before.
We have the parity decomposition $V_L=V_L^0  \oplus V_L^1$ in the standard way, and verify that this decomposition is compatible with the (super)bracket structure giving a vertex superalgebra structure on $V_L$. Then we form the vertex superalgebra $V_L \otimes F$, where $F$ is the free fermion vertex superalgebra generated by the field $\phi(x)=\sum_{n \in \mathbb{Z}} \phi(n) x^{-n-1/2}$.

We introduce (super)conformal structure on $V_L \otimes F$ by
choosing superconformal vector $\tau$ and conformal vector $\omega$ as follows:
$$\tau=\frac{\alpha(-1){\bf 1}}{\sqrt{pp'}} \otimes \phi(-1/2){\bf 1}+\frac{p-p'}{\sqrt{pp'}} {\bf 1} \otimes \phi(-3/2){\bf 1},$$
and
$$\omega=\frac{\alpha(-1)^2 {\bf 1}}{2pp'} \otimes {\bf 1} + (p-p')\frac{\alpha(-2){\bf 1}}{2pp'} \otimes {\bf 1}+  \frac{1}{2} \left( {\bf 1} \otimes \phi(-3/2)\phi(-1/2){\bf 1}\right).$$

It is easy to check that $Y(\tau,x)$ and $Y(\omega,x)$ close the $N=1$ Neveu-Schwarz superconformal algebra of central charge $\frac{3}{2}(1-2\frac{(p-p')^2}{pp'})$, and thus $V_L \otimes F$ is an $N=1$
vertex operator superalgebra. If we denote the eigenvalue of  $L(0)$ operator by $wt(v)$ then simple computation shows
$$wt(e^{l \alpha})=\frac{l^2 p p'-l(p-p')}{2}.$$ Now, let $L^{\circ}$ denote the dual lattice and $V_{L_0}$ generalized vertex operator superalgebra.
Therefore $e^{\alpha/p'} \in \mathbb{C}[L^\circ]$ and $e^{-\alpha/p} \in
\mathbb{C}[L^\circ]$ are odd vectors of weight $\frac{1}{2}$, and hence
$e^{\alpha/p'} \otimes \phi(-1/2){\bf 1}$ and $e^{-\alpha/p} \otimes
\phi(-1/2){\bf 1}$ are even vectors of weight one. It is not hard to
show (see \cite{AdM-striplet}) that the screening operators
$$Q={\rm Res}_{x} Y(e^{\alpha/p'} \otimes \phi(-1/2){\bf 1},x), \ \  \widetilde{Q}={\rm Res}_{x} Y(e^{-\alpha/p} \otimes \phi(-1/2){\bf 1},x)$$
commute with the action of the $N=1$ superconformal algebra.  We now define

$$\slogmin:= {\rm Ker}_{V_L \otimes F} (Q) \cap {\rm Ker}_{V_L \otimes F}(\widetilde{Q}).$$

From the previous discussion  we obtain
\begin{proposition}
We have
$$V_{1/2}(c_{p,p'},0) \subset  \slogmin, $$
thus $\slogmin$ is an $N=1$ superconformal vertex algebra.
\end{proposition}

We expect the vertex superalgebra $\slogmin$ to be $C_2$-cofinite. By using the construction analogous to the
one for $\logmin$ we can show that each $\slogmin$ admits logarithmic modules.

\begin{proposition}
\item[(i)] For  $i=0,...,p'-1$,
let $\mathcal{SV}_i=\left( V_{L+\frac{i}{p'} \alpha} \oplus  V_{L+\frac{i}{p'} \alpha-\frac{\alpha}{p}} \right) \otimes F$.
 Then $\mathcal{SV} = \mathcal{SV}_{0}$ is a vertex operator superalgebra and $\mathcal{SV}_i$ is a
 $\mathcal{SV}$--module.
\item[(ii)]  For  $j=0,...,p-1$, let $\mathcal{SV}^o_j= \left( V_{L+\frac{j}{p} \alpha} \oplus
 V_{L+\frac{j}{p} \alpha+\frac{\alpha}{p'}} \right) \otimes F$.
 Then $\mathcal{SV}^o = \mathcal{SV}^o _{0}$ is a vertex operator superalgebra and $\mathcal{SV}^o_j$ is a   $\mathcal{SV} ^o$--module
\end{proposition}

\begin{theorem} \label{log-modules-supertriplet}

\item[(i)] Assume that $(M, Y_M)$ is a weak
$\mathcal{SV}$--module. Then the structure  $$(\widetilde{M},
\widetilde{Y}_{\widetilde{M} }(\cdot,x) ) := ( M, Y_M(
\Delta(e^{-\a /p} \otimes \phi(-\hf),x) \cdot,x))$$

  is a weak $\slogmin$-module. Moreover,
$$(\widetilde{{\scV}_i}, \widetilde{Y}_{\widetilde{{\scV}_i} } ) ,
\quad  i=0, \dots,  p'-1, $$ are logarithmic $\slogmin$-modules.

\item[(ii)] Assume that $(M, Y_M)$ is a weak ${\scV}^{o}$--module.
Then the structure  $$(\widetilde{M}, \widetilde{Y}_{\widetilde{M}
}(\cdot,x)) := ( M, Y_M( \Delta(e^{\a /p'}\otimes \phi(-\hf),x)
\cdot,x))$$
%
%
%
  is a weak $ \slogmin$--module. In particular,
$$(\widetilde{{\scV}^o_j}, \widetilde{Y}_{\widetilde{{\scV}^o_j} } ) ,
\quad  j=0, \dots, p-1, $$ are logarithmic $\slogmin$--modules.
\end{theorem}


It is easy to modify the previous result to the case $p'=1$ and $p=2m+1$  considered in \cite{AdM-striplet}, but now with only one screening operator.
We shall use
the same notation as in \cite{AdM-striplet}. Let $L= {\Z}{\a}$ be
a rank one lattice with nondegenerate form given by
$ \la \a , \a \ra = 2m+1 $, where $m \in {\Zp}$. Let $V_L$ be the
corresponding vertex  superalgebra. Let $F$ be the fermionic vertex
algebra as above and $M$ its canonical twisted module as in
\cite{AdM-striplet} and \cite{AdM-tstriplet}. We consider the vertex
superalgebra $V_L \otimes F$. Let $\sigma$ be the canonical ${\Z}_2$
automorphism of $V_L \otimes F$.

%
%
%
%

%

The super-triplet vertex algebra is defined as
$$\striplet = \mbox{Ker}_{ V_L \otimes F} \widetilde{Q}. $$
It was proved in \cite{AdM-striplet} that $\striplet$ is an
irrational, $C_2$--cofinite vertex operator superalgebra. The
irreducible and irreducible $\sigma$--twisted $\striplet$--modules
were classified in \cite{AdM-striplet} and \cite{AdM-tstriplet}.

Now we shall construct certain logarithmic ($\sigma$--twisted)
modules for $\striplet$.
\begin{proposition}
\item[(i)] The vector space $\mathcal{SV}= ( V_L \oplus
V_{L-\a/(2m+1)}) \otimes F$ admits a structure of vertex operator
superalgebra.

\item[(ii)]  The vector space $\mathcal{RV}= ( V_{L+ \a /2} \oplus
V_{L + \a /2 -\a/(2m+1)}) \otimes M$ is a $\sigma$--twisted
$\mathcal{SV}$--module.
\end{proposition}

Let now $v = e^{-\a /(2m+1)} \otimes \phi(-\hf)$. Applying Theorem
\ref{gen-const-log} in the case of the vertex operator
superalgebra $\mathcal{SV}$ we get construction of logarithmic
modules for $\striplet$.

\begin{theorem}
\item[(i)] $ (\widetilde{\mathcal{SV}},
\widetilde{Y}_{\widetilde{\mathcal{SV}}})$ is a logarithmic
$\striplet$--module of semisimple length three.

\item[(ii)]  $(\widetilde{\mathcal{RV}},
\widetilde{Y}_{\widetilde{\mathcal{RV}}})$ is logarithmic
$\sigma$--twisted $\striplet$--module of semisimple length three.

\end{theorem}

Let us describe the twisted logarithmic module
$\widetilde{\mathcal{RV}}$ in more details. Let $(V_{\widetilde{L}},
Y)$ be the generalized vertex algebra associated to the lattice
$\widetilde{L} = {\Z} \frac{\a}{2 (2m+1)}$. As in
\cite{AdM-tstriplet}, let
$$\phi(x) = \sum_{ n \in {\Z}} \phi(n) x ^{-n - \thf}$$
be the field which defines on $M$ the structure of twisted
$F$--module. Then one can show that
$$\widetilde{Q} = \mbox{Res}_x Y(e^{-\alpha /(2m+1)},x) \otimes
\phi(z) = \sum_{n \in {\Z} } (e^{-\alpha /(2m+1)}) _{ -n -\thf}
\otimes \phi(n) $$
is the (second) screening operator which commutes with the action of
the Ramond algebra ( see Remark 5.2 of \cite{AdM-tstriplet}).
 By using similar analysis
as in Section \ref{triplet}, one can see that on  twisted module
$\widetilde{\mathcal{RV}}$:
$$ \widetilde{L(0)} = L(0) + \widetilde{Q}.$$

Then $\widetilde{\mathcal{RV}}$ is a ${\N}$--graded and that on its
top level $\widetilde{\mathcal{RV}} (0) $, the operator $\widetilde{L(0)}$ acts (in the obvious basis) as
$$\left[\begin{array}{cc} h^{2m,1} & 1 \\ 0 &   h^{2m,1} \end{array} \right], $$
where $h ^{2m,1} = \frac{1- 2m}{ 8} + \frac{1}{16}$. This gives
additional evidence for conjectures presented in
\cite{AdM-tstriplet}.

\section{Logarithmic representations of the affine vertex operator
algebra $\ver$ }

Admissible representations of affine Lie algebras  were introduced
by V. Kac and M. Wakimoto in \cite{KW}. These representations have
been  also studied in the framework of vertex operator algebras (cf.
\cite{A-1994}, \cite{AM-1995}, \cite{KW2}, \cite{P-2007}).

 In this section we shall present a
construction of logarithmic representations of the affine vertex
operator algebra ${\ver}$ which uses  realization from
\cite{A-2005}. This explicit realization enable us to apply
Theorem \ref{gen-const-log} in the case of the vertex operator
algebra ${\ver}$.

First we shall briefly recall the construction from \cite{A-2005}.

Define the following lattice
$$ \widetilde{L}= {\Z}\gamma + {\Z} \delta, \qquad \la \gamma , \gamma \ra = - \la
\delta , \delta \ra = \frac{1}{6}, \ \la \gamma , \delta \ra =
0.$$
Let $V_{ \widetilde{L}}$ be the associated generalized vertex
algebra with the vertex operator map $Y$.

Now we define the  screening operators:
\bea \label{exp-scr} &&Q = \mbox{Res}_x Y(e^{-6 \gamma},x) = e^{-6 \gamma}_0, \nonumber \\
&&\widetilde{Q} = \mbox{Res}_x Y(e^{2 \gamma},x)= e^{2 \gamma} _0,
\nonumber  \eea

and the following Virasoro element
$$ \omega = (3 \gamma(-1)^{2} - 2 \gamma (-2) -
3 \delta(-1)^{2}){\bf 1}. $$
Then $\omega$  generates the Virasoro vertex operator algebra
$L^{Vir}(-6,0)$. Set
$$L(x)= Y(\omega,x) = \sum_{ n \in {\Z} } L(n) x^{-n-2}. $$

  Define the following vectors in $ V_{\widetilde{L}}$:
\bea
& e &= e^{3 (\gamma - \delta)}, \nonumber \\
& h & = 4 \delta (-1){\bf 1} , \nonumber \\
& f & = -\tfrac{2}{9} Q e^{3 (\gamma + \delta)} \nonumber \\ && =
- ( 4 \gamma(-1) ^{2} - \tfrac{2}{3}\gamma(-2) ) e^{-3(\gamma -
\delta)}. \nonumber
\eea

Then $e, h, f$ are primary vectors of conformal weight $1$ for the
Virasoro algebra.
%

Define
$$  {D}= {\Z}(3\gamma + 3\delta) + {\Z}(3\gamma - 3\delta).$$
Clearly, $V_{ {D}}$ is a subalgebra of $V_{\widetilde{L}}$ which
contains vectors $e$, $f$ and $h$. Since $D$ is an even lattice,
$V_D$ is a vertex algebra.

\begin{theorem} \cite{A-2005} The vectors $e$, $f$ and $h$ generate a subalge\-bra of the
vertex algebra $V_{{D}}$ isomorphic to the vertex operator algebra
${\ver}$. Moreover,
$$\ver \subset \mbox{Ker}_{V_D} \widetilde{Q}. $$
\end{theorem}

Now we shall consider the vertex operator algebra $${\cV} = V_D
\oplus V_{D + 2 \gamma}$$ and its simple module
$$ \mathcal{M} = V_{ D-3\gamma + \delta} \bigoplus V_{ D -\gamma +
\delta}.$$
(The associated vertex operators are restrictions of the vertex
operator in the generalized vertex algebra $V_{  \widetilde{L} }$
and will be also denoted by $Y$ ).

We also note that $V_{D - 2 \gamma}$ and $V_{D-\gamma + \delta}$ are
(irreducible)  ${\cV}$--modules (on which $V_{D + 2 \gamma}$ acts
trivially) and that there is a non-trivial intertwining operator
${\mathcal Y}$ of type
\bea {\mathcal{M} \choose  V_{D-\gamma + \delta} \quad V_{D-2 \gamma
} }.  \label{sl2-int-1}\eea

 Let now $v = e^{2 \gamma} $. Applying Theorem
\ref{gen-const-log} in the case of the vertex operator
superalgebra $\mathcal{V}$ we get construction of logarithmic
${\ver}$--modules.

\begin{theorem}
 The pairs $  (\widetilde{{\cV}}, \widetilde{Y} )$ and $ (\widetilde{\mathcal{M}}, \widetilde{Y})$ are
 logarithmic ${\ver}$--modules. In particular, the components of
 the fields
 \bea
&& \widetilde{e(x)} = \widetilde{Y}(e,x) = Y(e,x), \nonumber \\
&& \widetilde{h(x)} = \widetilde{Y}(h,x) = Y(h,x), \nonumber \\
&&\widetilde{f(x)} = \widetilde{Y}(f,x) = Y(f,x) + \frac{4}{3} x
^{-1} Y( \gamma(-1) e^{-\gamma + 3 \delta},x) + \frac{1}{3} x^{-2}
Y( e^{-\gamma + 3 \delta},x), \nonumber
 \eea
define on ${\cV}$ and $\mathcal{M}$  structures of  $A_1
^{(1)}$--modules of level $-\tfrac{4}{3}$.
\end{theorem}

Define now the following module

$$ \mathcal{R}_{-1/3} = U (\widehat{sl_2}) . e^{-3 \gamma +
\delta} \subset \widetilde{\mathcal{M}}. $$
Recall that
$$ \widetilde{L(0)} =  L(0) + \widetilde{Q}.$$
Then
\bea
&&\widetilde{L(0)} e^{-3 \gamma + \delta } =  -\frac{1}{3} e^{-3
\gamma + \delta }
+ e^{-\gamma + \delta}, \nonumber \\
&&\widetilde{L(0)}e^{-\gamma + \delta} =  -\frac{1}{3} e^{-\gamma
+ \delta}. \nonumber
 \eea

 Therefore $\mathcal{R}_{-1/3} $ is an logarithmic $\ver$--module.

On the other hand, the results from \cite{A-2005} implies that
 $$U (\widehat{sl_2}) . e^{-  \gamma +
 \delta} \cong L(-\frac{2}{3} \Lambda_0 - \frac{2}{3} \Lambda_1), $$
and that  $$ E =U (\widehat{sl_2}) . e^{-  2 \gamma}  $$ is a
$\N$--graded
 ${\ver}$--module.

 Note that  $\mathcal{R}_{-1/3} $ contains $e^{-  \gamma +
 \delta}$ and therefore a submodule isomorphic to $ L(-\frac{2}{3} \Lambda_0 - \frac{2}{3}
 \Lambda_1)$.
 But  we shall now  see that $\mathcal{R}_{-1/3} $ is not $\N$--graded.
 We have:

 \bea &&\widetilde{e(n+1)} e^{-3 \gamma + \delta } = \widetilde{f(n+2)} e^{-3 \gamma + \delta }
 = 0 \quad \mbox{for} \ n \ge 0 , \label{hw-1}\\
 && \widetilde{e(0)} e^{-3 \gamma + \delta } = e ^{-2 \delta}, \quad \widetilde{f(1)} e^{-3 \gamma + \delta
 } = \nu e ^{-4 \gamma + 4 \delta}  \quad (\nu \ne 0).
 \label{hw-2} \eea

Thus $e ^{-4 \gamma + 4 \delta} \in \mathcal{R}_{-1/3} $.

 By using   explicit realization \cite{A-2005}, one can see  that
 $$ E_1 =U (\widehat{sl_2}) . e^{- 4 \gamma +
4 \delta} $$
is a weak, ordinary module $\Z$--graded module, but gradation
is not bounded. (Note that also  $L(-\frac{2}{3} \Lambda_0 -
\frac{2}{3} \Lambda_1)$ is isomorphic to a submodule of $E_1$).
Now relations (\ref{hw-1})-(\ref{hw-2}) imply that
$$E_2 =  \mathcal{R}_{-1/3} / U (\widehat{sl_2}) . e^{- 4 \gamma +
4 \delta}$$ is a non-logarithmic ${\N}$--graded $\ver$--module.
\begin{proposition}
The logarithmic ${\ver}$--module $\mathcal{R}_{-1/3}$ appears in
the extension
$$ 0 \rightarrow E_1
\rightarrow \mathcal{R}_{-1/3}  \rightarrow E_2 \rightarrow 0 $$
of non-logarithmic weak module $E_1 $ and non-logarithmic
${\N}$--graded module $E_2$.
\end{proposition}

The results from this section imply the existence of the following
intertwining operator (this intertwining operator can be also
constructed from intertwining operator (\ref{sl2-int-1}) and methods
from Section \ref{log-int-sec} below).
\begin{proposition} \label{sl2-int-2}
There is a non-trivial (logarithmic) intertwining operator of type
$${ \mathcal{R}_{-1/3} \choose L(-\frac{2}{3} \Lambda_0 -
\frac{2}{3} \Lambda_1) \quad E }. $$
\end{proposition}

\begin{remark}
We should mention that the existence of such logarithmic module
$\mathcal{R}_{-1/3}$ was predicted  in the fusion rules analysis of
\cite{G}.  Gaberdiel found ${ \mathcal{R}_{-1/3} } $ in the fusion
product (which should correspond to tensor product) of
non-logarithmic $\N$-graded modules $L(-\frac{2}{3} \Lambda_0 -
\frac{2}{3} \Lambda_1)$ and $E$. Our intertwining operator from
Proposition \ref{sl2-int-2} is in agreement with this fusion product
in \cite{G}.

 One can
also check that the structure of our logarithmic module
$\mathcal{R}_{-1/3}$ can be also illustrated by Figure 2 in
\cite{G}. So we  believe that our construction present an explicit
realization of this logarithmic module.

We can also make a similar analysis of other logarithmic
submodules of
 $  \widetilde{{\cV}}$  and $ \widetilde{\mathcal{M}}$.  We plan
 to study these logarithmic modules and the corresponding fusion rules in our forthcoming papers.
\end{remark}

\section{Logarithmic intertwining operators}

\label{log-int-sec}  Let us recall the main ingredients in the definition of logarithmic
intertwining operators following \cite{M1}, \cite{M2} and \cite{HLZ}. These are simply intertwining-like operator maps
$$\mathcal{Y}(w,x) : W_1 \otimes W_2 \longrightarrow W_3 \{x \} [{\rm log}(x)],$$
satisfying the $L(-1)$-property and the Jacobi identity identical to the ordinary Jacobi identity in the definition of ordinary intertwining operator
maps. These objects appear naturally if one studies differential equations satisfied by correlation functions \cite{M1}.
First nontrivial constructions of intertwining operators were obtained in \cite{M2} and \cite{AdM-2007} (these structures were predicted earlier in \cite{M1}).
In this part we present a general construction of intertwining operators among certain logarithmic modules.
The construction here is heavily influenced by results in previous sections  and especially \cite{FFHST}, were closely related operators were considered.

We shall use the same framework as in Section \ref{konstrukcija}.
Let $V$ be a vertex operator superalgebra. Let again $v \in V$
satisfies conditions (\ref{rel-c-1}) and (\ref{rel-c-2}).  Denote by
$\overline{V}$ any vertex subalgebra of $V$  contained in
the kernel of $v_0$.  If $M$ is a weak $V$-module, then the pair
$(\widetilde{Y}_{\widetilde{M}},\widetilde{M})$ defines a weak
$\overline{V}$-module.

  Since $M$ is an $V$--module, we have an intertwining operator, which we denote by
  $\mathcal{Y}$, of type
$${ M   \choose V  \ \  M }.$$

We are aiming for a logarithmic intertwining operator of type
$${ \widetilde{ M}  \choose V  \ \  \widetilde{M} }.$$

Recall
$$ \Delta(v,x) = x^{v_0} \exp \left( \sum_{n=1} ^{\infty}
\frac{v_n}{-n}(-x)^{-n} \right)$$ be as before. Observe (again!) that the
expression
$$\Delta(v,x)w, \ \ w \in V,$$
is ambiguous in general, but
$$\Delta(v,x)w, \ \ w \in \overline{V},$$
is well-defined. To fix the problem we introduce a new operator (cf.
\cite{FFHST} and \cite{M2} for closely related constructions): \be
\label{newdelta} \Delta_{log}(v,x)={\rm exp}\left({{\rm log}(x)
v_0}+  \sum_{n=1} ^{\infty} \frac{v_n}{-n}(-x)^{-n} \right), \ee
where as usual
$${\rm exp}({\rm log}(x)v_0)=\sum_{n = 0}^\infty  \frac{{\rm log}(x)^n v_0^n}{n!},$$
and ${\rm log}(x)$ is a formal variable satisfying $({\rm log}(x))'=\frac{1}{x}$.

Notice that now  $\Delta_{log}(v,x)w$ involves logarithmic terms for
$w \in V$, but if $w \in \overline{V}$ then \be \label{key-relation}
\Delta(v,x)w=\Delta_{log}(v,x)w. \ee

We first observe the following useful result:

\begin{lemma} \label{key} With $u \in \overline{V}$ and $v$ as above, then
$$\Delta_{log}(v,x_2)Y(u,y)=Y(\Delta(x_2+y)u,y)\Delta_{log}(v,x_2).$$
\end{lemma}
\noindent \begin{proof} Because of $[v_0,v_m]=0$
we can write
$$\Delta_{log}(v,x)={\rm exp}({\rm log}(x)v_0)E^+(-v,-x),$$
where $E^+(v,x)$ is defined as usual.

As in Lemma 2.15 of \cite{Li-sc} we get
$$E^+(v,x)Y(u,y)E^+(-v,x)=Y((-x)^{v_0}\Delta(-v,-x+y)u,y).$$
As in \cite{M1}, we also have
$${\rm exp}({\rm log}(x)v_0)Y(u,y){\rm exp}(-{\rm log}(x)v_0)=Y({\rm exp}({\rm log}(x) v_0) u,y).$$
Combined together we get \bea && \Delta_{{\rm
log}}(v,x_2)Y(u,y)\Delta_{{\rm log}}(-v,x_2) \nn &&={\rm exp}({\rm
log}(x)v_0)E^+(-v,-x)Y(u,y)E^+(v,-x){\rm exp}(-{\rm log}(x)v_0)
\nn &&={\rm exp}({\rm log}(x)v_0) Y(x^{-v_0}\Delta(v,x+y)u,y){\rm
exp}(-{\rm log}(x)v_0) \nn &&=Y( {\rm exp}({\rm log}(x)v_0)
x^{-v_0} \Delta(v,x+y)u,y). \eea Finally, if we take $v \in
\overline{V}$, we obtain the wanted formula. \end{proof}

Here is the main result of this section

\begin{theorem} \label{log-int-1}Let $\mathcal{Y}$ be intertwining operator as above and $v$ as before such that $v_0$ acts nilpotently on $V$ . Then
$\widetilde{\mathcal{Y}}( \cdot,x):=\mathcal{Y}(\Delta_{log}(v,x)
\cdot,x)$ is a logarithmic intertwining operator of type
$${ \widetilde{ M}  \choose V  \ \  \widetilde{ M} }$$
in the category of $\overline{V}$--modules.
\end{theorem}

\noindent \begin{proof} The main idea is already in \cite{Li-sc}.
Because of
$$[L(-1),v_0]=0, \ \ [L(-1),v_n]=-nv_{n-1},$$
we have
$$[L(-1), \Delta_{log}(u,x)]=-\frac{d}{dx}(\Delta_{log}(v,x)),$$
which implies
$$[L(-1),\tilde{\mathcal{Y}}(u,x)]=\frac{d}{dx}\tilde{\mathcal{Y}}(u,x),$$
the $L(-1)$-derivative property.

For the Jacobi identity, we start from the ordinary Jacobi identity
for the intertwining operator $\mathcal{Y}(u,x)$, where $u \in V $:

$$x_0^{-1} \delta \left( \frac{x_1-x_2}{x_0} \right)  \tilde{Y}(u,x_1)  \tilde{\mathcal{Y}}(w,x_2)-x_0^{-1} \delta \left( \frac{-x_2+x_1}{x_0} \right) \tilde{\mathcal{Y}}(w,x_2) \tilde{Y}(u,x_1)$$
$$=x_0^{-1} \delta \left( \frac{x_1-x_2}{x_0} \right)  {Y}(\Delta(v,x_1)u,x_1) {\mathcal{Y}}(\Delta_{log}(v,x_2)w,x_2)$$
$$-x_0^{-1} \delta \left( \frac{-x_2+x_1}{x_0} \right) {\mathcal{Y}}( \Delta_{log}(v,x_2)w,x_2) {Y}(\Delta(v,x_1)u,x_1)$$

$$=x_2^{-1} \delta \left(\frac{x_1-x_0}{x_2} \right) \mathcal{Y}(Y(\Delta(v,x_1)u,x_0) \Delta_{log}(v,x_2)w,x_2)$$
$$=x_2^{-1} \delta \left(\frac{x_1-x_0}{x_2} \right) \mathcal{Y}(Y(\Delta(v,x_2+x_0)u,x_0) \Delta_{log}(v,x_2)w,x_2).$$

Finally, we apply Lemma \ref{key} so the right hand-side becomes

$$x_2^{-1} \delta \left(\frac{x_1-x_0}{x_2} \right) \mathcal{Y}(\Delta_{log}(v,x_2) Y(u,x_0)w,x_2),$$
which is what we needed to prove.

\end{proof}

By using similar proof to that of   Theorem \ref{log-int-1} we have

\begin{theorem} \label{log-int-2} Assume that $M_1$, $M_2$ and $M_3$ are $V$--modules and  that  $v \in V$ satisfies conditions
(\ref{rel-c-1}) and (\ref{rel-c-2}) and $v_0$ acts nilpotentlly on $M_i$, $i=1,2,3$. Let $ \mathcal{Y}$ be an
intertwining operator of type ${ M_3 \choose M_1 \ \ M_2 }$ in the
category of $V$--modules. Then  $\widetilde{\mathcal{Y}}(
\cdot,x):=\mathcal{Y}(\Delta_{log}(v,x) \cdot,x)$ is a logarithmic
intertwining operator of type
$${ \widetilde{ M_3}  \choose M_1 \ \  \widetilde{ M_2} }$$
in the category of $\overline{V}$--modules.
\end{theorem}

\noindent {\large {\bf Acknowledgments}} \\ D.A. would like to thank
the Erwin Schr$\ddot{\mbox{o}}$dinger Institute in Vienna for
hospitality. We also thank anonymous referee for his/her
constructive comments.   The second author was partially supported
by NSF grant DMS-0802962.


\begin{thebibliography}{EF}

\bibitem[1]{Abe} T. Abe,
A ${\Bbb Z}\sb 2$-orbifold model of the symplectic fermionic vertex
operator superalgebra. {\em Math. Z.} {\bf 255} (2007), 755--792.

\bibitem[2]{A-1994} D. Adamovi\'{c},
Some rational vertex algebras, Glas. Mat. Ser. III {\bf 29}(49)
(1994), 25-40.

\bibitem[3]{A-2003} D. Adamovi\'{c},
Classification of irreducible modules of certain subalgebras of free
boson vertex algebra, J. Algebra {\bf  270} (2003) 115--132.

\bibitem [4]{A-2005} D. Adamovi\' c, A construction of admissible
$A_1^{(1)}$--modules of level $-\tfrac{4}{3}$, J. Pure  Appl.
Algebra {\bf 196} (2005) 119-134.


\bibitem [5]{AM-1995} D. Adamovi\' c and A. Milas,
Vertex operator algebras associated to the modular invariant
representations for $A_1^{(1)}$, Math. Res. Lett. {\bf 2} (1995) 563-575.

\bibitem[6]{AdM-2007} D. Adamovi\'c and A. Milas, Logarithmic intertwining operators
and $\mathcal{W}(2,2p-1)$-algebras, {\em Journal of Math. Physics}
{\bf 48}, 073503 (2007).

\bibitem[7]{AdM-triplet} D. Adamovi\'c and A. Milas, On the triplet vertex algebra
$\mathcal{W}(p)$, {\em Advances in Math.} {\bf 217} (2008) 2664--2699;
arxiv:0707.1857.

\bibitem[8]{AdM-striplet} D. Adamovi\'c and A. Milas, The $N=1$
 triplet vertex operator superalgebras, Comm. Math.
Phys. {\bf 288} (2009), 225–-270; arXiv:0712.0379.

\bibitem[9]{AdM-tstriplet}   D. Adamovi\'c and A. Milas, The $N=1$
 triplet vertex operator superalgebras: twisted sector, SIGMA {\bf 4} (2008) 24 pages,
 arXiv:0806.3560.

\bibitem[10]{AdM-ptraces} D. Adamovi\'c and A. Milas, An analogue of modular BPZ-equations in logarithmic (super)conformal field theory,
 Vertex Operator Algebras and Related Areas Edited by: Maarten
Bergvelt, Gaywalee Yamskulna and Wenhua Zhao, Normal, IL 2009;
Contemporary Mathematics {\bf  497}, 17 pages.

\bibitem[11]{AdM-wp2} D. Adamovi\'c and A. Milas,
On W-algebras associated to $(2, p)$ minimal models and their representations, preprint; arXiv:0908.4053.

\bibitem[12]{Ar}  T.  Arakawa, Representation theory of superconformal algebras and the Kac-Roan-Wakimoto conjecture,
 {\em Duke Math. J.}  {\bf 130}  (2005), 435-478.

\bibitem[13]{CF} N. Carqueville and M. Flohr,  Nonmeromorphic operator product expansion
and $C\sb 2$-cofiniteness for a family of $\mathcal{W}$-algebras.
{\em J. Phys.} A {\bf 39} (2006), 951--966.

\bibitem[14]{DK} A. De Sole and V. Kac, Finite vs. affine
$W$-algebras, {\em Japanese Journal of Math.} {\bf 1} (2006)
137-261; arxiv:math/05110055.

\bibitem [15]{DLM-1996} C. Dong, H. Li,  and  G. Mason,  Simple currents and extensions of
vertex operator algebras. Comm. Math. Phys. {\bf 180} (1996) 671–-707.

\bibitem[16]{Flohr-mod} M. Flohr, On modular invariant partition functions of conformal field theories with logarithmic operators,
{ Internat. J. Modern Phys.} A {\bf 11 } (1996), 4147--4172.

 \bibitem[17]{FFHST} J. Fjelstad, J. Fuchs, S. Hwang, A.M.
Semikhatov and I. Yu. Tipunin, Logarithmic conformal field theories
via logarithmic deformations, {  Nuclear Phys.} B {\bf 633} (2002),
379--413.

\bibitem[18]{FGST-triplet} B.L. Feigin, A.M. Ga\u\i nutdinov, A. M. Semikhatov, and I. Yu Tipunin,
Modular group representations and fusion in logarithmic conformal
field theories and in the quantum group center. {\em Comm. Math.
Phys.} {\bf 265} (2006), 47--93.

\bibitem[19]{FGST-triplet2} B.L. Feigin, A.M. Ga\u\i nutdinov, A. M. Semikhatov, and I. Yu Tipunin, Kazhdan--Lusztig correspondence for the representation category of the triplet W-algebra in logarithmic CFT, Theor.Math.Phys. 148 (2006) 1210-1235; Teor.Mat.Fiz. {\bf 148} (2006) 398-427.

\bibitem [20]{FGST-log} B.L. Feigin, A.M. Ga\u\i nutdinov, A. M. Semikhatov, and I. Yu Tipunin,
 Logarithmic extensions of minimal
models: characters and modular transformations, Nucl. Phys. B {\bf 757}
(2006) 303–343.



\bibitem[21]{Fu} J. Fuchs,On non-semisimple fusion rules and tensor categories,
Contemporary Mathematics {\bf 442} (2007) 315-337, arXiv:hep-th/0602051.


\bibitem[22]{FHST} J. Fuchs, S. Hwang, A.M. Semikhatov and I. Yu. Tipunin,
Nonsemisimple Fusion Algebras and the Verlinde Formula, { Comm.
Math. Phys.} {\bf 247} (2004), no. 3, 713--742.



\bibitem[23]{G} M. Gaberdiel, Fusion rules and logarithmic
representations of a WZW model at fractional level, Nuclear Phys. B
{\bf 618} (2001) 407-436.

\bibitem[24]{Ga} M. Gaberdiel, An algebraic approach to logarithmic conformal field theory,
 Proceedings of the School and Workshop on Logarithmic Conformal Field Theory and its Applications (Tehran, 2001),
{ Internat. J. Modern Phys.} A {\bf 18} (2003), 4593--4638.

\bibitem[25]{GaKa} M. Gaberdiel and H. G. Kausch, A rational logarithmic
conformal field theory, Phys. Lett B {\bf  386} (1996), 131-137,
hep-th/9606050


\bibitem[26]{Hu} Y.-Z. Huang,  Cofiniteness conditions, projective covers and the logarithmic tensor product theory,
J. Pure Appl. Algebra {\bf 213} (2009) 458-475;
 {\tt arxiv.0712.4109}.

\bibitem[27]{HLZ}  Y.-Z. Huang, J. Lepowsky and L. Zhang,  Logarithmic tensor product theory for generalized modules for a conformal vertex algebra, {\tt arXiv:0710.2687}.


\bibitem[28]{Kac}  V. Kac, {\em Vertex algebras for beginners}, University Lectures Series,
Vol. 10, Providence, 1998.

\bibitem[29]{KW} V. Kac and M. Wakimoto, Modular invariant representations of
infinite dimensional Lie algebras and superalgebras, Proc. Natl.
Acad. Sci. USA, Vol. {\bf 85} (1988) 4956-4960.

\bibitem[30]{KW2} V. Kac and M. Wakimoto, On rationality of
W-algebras, Transform. Groups {\bf 13}, No. 3-4 (2008) 671-713,  {\tt arXiv:0711.2296}.


\bibitem[31]{Ka} H. G. Kausch, Extended conformal algebras
generated by multiplet of primary fields, { Phys. Lett.} {\bf 259} B
(1991), 448-455.

\bibitem[32] {LL} J. Lepowsky and H. Li, {\em Introduction to Vertex Operator
Algebras and Their Representations}, Progress in Mathematics, Vol.
{\bf 227}, Birkh\"auser, Boston, 2003.

\bibitem [33]{Li-bilinear}
H. Li, Symmetric invariant bilinear forms on vertex operator
algebras, J. Pure Appl. Algebra {\bf 96} (1994), 279–297.

\bibitem[34]{Li-sc} H. Li, The physics superselection principal in vertex operator
algebra theory, J. Algebra {\bf 196} (1997), 436–457.

\bibitem[35]{M1} A. Milas, Weak modules and logarithmic intertwining operators for vertex operator algebras. Recent developments in infinite-dimensional Lie algebras and conformal field theory (Charlottesville, VA, 2000),
201--225, {\em Contemp. Math.} {\bf 297}, Amer. Math. Soc.,
Providence, RI, 2002.

\bibitem[36]{M3} A. Milas, Fusion rings for degenerate minimal models, { J. Algebra} {\bf 254} (2002), no. 2, 300--335.

\bibitem[37]{M2} A. Milas, Logarithmic intertwining operators and
vertex operators, {\em Comm. Math. Phys.}  {\bf 277} (2008),
497-529; {\tt math.QA/0609306}.


\bibitem[38]{Miy} M. Miyamoto,
Modular invariance of vertex operator algebras satisfying $C\sb
2$-cofiniteness. {\em Duke Math. J.} {\bf 122} (2004), 51--91.


\bibitem[39]{P-2007} O. Per\v se, Vertex operator algebras associated to type $B$ affine
Lie algebras on admissible half-integer levels, J. Algebra {\bf
307}  (2007) 215-238.

\end{thebibliography}
\end{document}